\documentclass[12pt,reqno]{amsart}
\usepackage{pgfplots}
\usepackage{tikz}
\pgfplotsset{compat=newest}
\usepackage{color}
\makeatletter
\renewcommand*{\@cite}[2]{\fcolorbox{black}{white}{#1\if@tempswa, #2\fi}}
\renewcommand*{\@biblabel}[1]{{\fcolorbox{green}{white}{#1}}\hfill}
\makeatother
\usepackage{amsmath,amsfonts,amssymb,amsthm}
\usepackage[normalem]{ulem}
\usepackage{graphicx}
\graphicspath{ }
\usetikzlibrary{intersections}
\usetikzlibrary{patterns}
\usepackage{soul}
\usepackage{lipsum}
\usepackage{epstopdf}
\usepackage{pdflscape}
\usepackage{csquotes}
\usepackage{wrapfig}
\usepackage{accents}
\usepackage{adjustbox}
\usepackage{tikz-3dplot}
\usepackage{caption}
\usepackage{subcaption}
\usepackage{calligra}
\usepackage{xcolor}
\usepackage{etoolbox}
\usepackage[colorlinks]{hyperref}
\hypersetup{backref=true,
	urlcolor=blue,
	linkcolor=blue,
	bookmarks=true,
	filecolor=black,
	citecolor=red,
	citebordercolor=green,
	filebordercolor=red,
	linkbordercolor=green
}

\numberwithin{figure}{section}

\theoremstyle{plain}

\newtheorem{thm}{Theorem}[subsection]

\newtheorem{lem}[thm]{Lemma}
\newtheorem{prop}[thm]{Proposition}
\newtheorem{cor}{Corollary}[thm]
\theoremstyle{definition}
\newtheorem{defn}[thm]{Definition}

\newtheorem{exmp}[thm]{Example}
\newtheorem{rem}[thm]{Remark}
\newtheorem{question}{Question}

\numberwithin{equation}{section}
\usepackage{mathtools}

\makeatletter
\@namedef{subjclassname@2020}{%
	\textup{2020} Mathematics Subject Classification}
\makeatother

\title{Rigidity and unlikely intersection in higher-dimensional formal groups}
\author[R. Abdellatif]{Ramla Abdellatif}
\address{Laboratoire Ami\'enois de Math\'ematique Fondamentale et Appliqu\'ee, Universit\'e de
Picardie Jules Verne, 33, rue Saint-Leu - 80 039 Amiens Cedex 1, France}
\email{ramla.abdellatif@math.cnrs.fr}

\author[M. A. Sarkar]{MABUD ALI SARKAR}
\address{Mabud Ali Sarkar, Department of Mathematics, The University of Burdwan \& Nagaland University, India.}
\curraddr{Mabud Ali Sarkar, Department of Mathematics, Darjeeling Hills University, \newline Darjeeling-73413, India.}
\email{mabudji@gmail.com}

\author[A. A. Shaikh]{Absos Ali Shaikh}
\address{Absos Ali Shaikh, Department of Mathematics, The University of Burdwan, \newline Burdwan-713101, India.}
\email{aashaikh@math.buruniv.ac.in}

\begin{document}
	
\begin{abstract}
In this article, we identify a class of higher-dimensional formal groups over the ring of $p$-adic integers that are uniquely determined by their $p$-power torsion points. More precisely, we prove that if two simple finite-height formal groups share infinitely many torsion points, then they are equal. This extends a rigidity theorem of Berger \cite{LB1} from the one-dimensional setting to a higher-dimensional family of simple formal groups.
\end{abstract}

	\maketitle
	\tableofcontents 
	\section{Introduction and Motivation}\label{s1}

	The study of the torsion points of formal groups has achieved significant attention over the past few decades due to its relevance in number theory. Lubin and Tate \cite[1964]{JL2} pioneered the groundwork by establishing a remarkable connection to local class field theory, which has since been expanded upon by several authors, including but not limited to Iwasawa \cite{KI}, De Shalit \cite{EDS}, Berger \cite{LB}, Rosen and Zimmermann \cite{MR}, Matson \cite{CM}, Abdellatif and Sarkar \cite{RM}. Recently, Iovita, Morrow, and Zaharescu \cite{AIV2} constructed a tamely ramified extension over the completion of the maximal unramified extension of $\mathbb{Q}_p$ by employing the $p$-power torsion points of the formal group of an abelian variety. Berger \cite{LB1} showed that under a certain condition, the torsion points of a formal group can determine the formal group itself. Some similar results were reported by Bogomolov and Tschinkel \cite[Chap.~4]{BT} for elliptic curves, whereas by Baker and DeMarco \cite{MBD} for rational functions.
	
	In this report, we extend the work of Berger \cite{LB1} to higher dimensional formal groups. To be precise:
	
	Let $\mathbb{Z}_p$ be the ring of integers of the $p$-adic field $\mathbb{Q}_p$ and $\mathfrak{m}_{\mathbb{C}_p}$ be the unit disk in $\mathbb{C}_p:=\widehat{\bar{\mathbb{Q}}_p}$. Let $K$ be a finite extension of $\mathbb{Q}_p$ and denote its ring of integers by $\mathcal{O}_K$. 
    
    Let $F(X,Y)=X+Y+h.o.t. \in \mathcal{O}_K[[X,Y]]$ be a one-dimensional formal group law over $\mathcal{O}_K$. Let $\text{Tors}(F)$ be the set of torsion points of $F$ in $\mathfrak{m}_{\mathbb{C}_p}=\{z \in \mathbb{C}_p~|~|z|_p<1\}$. Then
	Berger proved:
	\begin{thm} \cite[Theorem~A]{LB1} \label{t1.1}
		If $F$ and $G$ are one-dimensional formal groups over $\mathcal{O}_K$ such that $\text{Tors}(F) \cap \text{Tors}(G)$ is infinite, then $F=G$.
	\end{thm} 
	The preceding result is devoted to one-dimensional formal groups. In this report, we show that the above result generalizes to higher-dimensional formal groups which are simple:
	\begin{thm} \label{t3.6}
		If $F$ and $G$ are $d$-dimensional simple formal groups over $\mathbb{Z}_p$ such that $\text{Tors}(F) \cap \text{Tors}(G)$ is infinite, then $F=G$.
	\end{thm}

	\section{Preliminaries}\label{s3}
	
  In this section, we recall the notation and basic concepts that will be used throughout the paper. For background on higher-dimensional formal groups, their endomorphisms, and related topics, we refer the reader to \cite{RM,MH}. We therefore restrict ourselves to those definitions and results that are directly relevant to our work. We begin by introducing the class of formal groups that will be the primary object of study.
    
	\begin{defn} \label{d2.7}
		A formal group $F$ is called \emph{simple as a group object} if it admits no nontrivial formal
		subgroup. We say that $F$ is \emph{simple as a formal group law} if it cannot be expressed as a
		direct sum of formal groups of strictly smaller dimension.
	\end{defn}

	\begin{exmp} \label{exam2.2}
		A formal group over a field of characteristic zero, which gives rise to a simple Lie algebra, will be an example of a simple formal group. It is worth noting that every formal group gives rise to a Lie algebra with the help of the coefficients of the quadratic terms of its formal power series expansion. Moreover, over a field of characteristic 0, the category of finite dimensional formal groups is equivalent to the category of finite dimensional Lie algebras, while the property \enquote{simple} is preserved under equivalent functor.
	\end{exmp}

	\begin{defn}
		If $F$ is a $d$-dimensional formal group over $\mathbb{Z}_p$ with a multiplication-by-$p$ endomorphism $[p]_F(x)$, then its torsion points are the analytic zeros of the iterates of $[p]_F(x)$ in the rigid analytic disc $\mathfrak{m}_{\mathbb{C}_p}^d$, defined by $$\text{Tors}(F):=\bigcup_{n \in \mathbb{N}}\{\xi \in \mathfrak{m}_{\mathbb{C}_p}^d~|~[p^n]_F(\xi)=0 \},~\xi=(\xi_1,\cdots,\xi_d).$$
	\end{defn}
	If $F$ is of finite height $h$, then $\text{Tors}(F)$ is infinite. Moreover, $\text{Tors}(F)$ is Zariski dense in $\mathfrak{m}_{\mathbb{C}_p}^d$.
	
	Here, the height of a $d$-dimensional formal group is defined as follows:
	\begin{defn} \cite[Sec.~18.3.8]{MH} \label{d1.3}
		Let $F(X,Y)$ be an $d$-dimensional formal group $F(X,Y)$ over the ring of $p$-adic integers $\mathbb{Z}_p$ with endomorphism $[p]_F(X)=(f_1(X), \cdots, f_d(X))$ (mod $p$). If the ring $\mathbb{F}_p[[x_1,\cdots,x_d]]$ is a finitely generated and free module over the subring $\mathbb{F}_p[[f_1(X), \cdots, f_d(X)]]$ of rank $p^h, \ h \in \mathbb{N}$, then $h$ is the height of $F$. We use the notation $\mathbb{F}_p$ for finite field with $p$ elements.
	\end{defn}
\begin{defn}
		The Jacobian of a power series $h(X)=(h_1(X), \cdots, h_{d_1}(X)) \in \mathbb{Q}_p[[X]]^{d_1}$, $X=(x_1, \cdots, x_{d_2})$ is a $d_1 \times d_2$ matrix defined by 
		\[
		J(h):=\left[\frac{\partial}{\partial x_j}h_i \right],~1 \leq i \leq d_1,~1 \leq j \leq d_2.
		\]
		
	\end{defn}
	In this article, we often denote the Jacobian of such a function $h$ at $0$ by $J_0(h):=J(h)(0)$.

	\begin{prop} \label{pp2.8}
		The sufficient condition for $h \in \mathcal{O}_K[[X]]^d,~X=(x_1, \cdots,x_d)$ with $h(0)=0$ being invertible is that $J_0(h) \in \text{GL}_d(\mathcal{O}_K^{\times})$.
	\end{prop}
	\begin{proof}
		The proof is straightforward. See Appendix \ref{A1}.
	\end{proof}

	The above result says if $h(X) \in \mathbb{Z}_p[[X]]^d$ with $h(0)=0$ and $J_0(h) \in \text{GL}_d(\mathbb{Z}_p^{\times})$, then $h^{\circ -1}(X) \in \mathbb{Z}_p[[X]]^d$ with $h^{\circ -1}(0)=0$.  

We also require the notion of \textit{stable} power series in higher dimension, which was originally introduced by Lubin \cite{JL} for single variable power series.
	\begin{defn} \label{d4.1}
		A power series $h(X) \in \mathbb{Q}_p[[X]]^d,~X=(x_1, \cdots, x_d)$ with $h(0)=0$ is said to be \textit{stable} if the eigenvalues of $J_0(h)$ are neither zero nor roots of unity.  
	\end{defn}
	\begin{exmp}
		If $u(X) \in \mathbb{Z}_p[[X]]^2, ~X=(x_1,x_2)$ be defined by 
		\begin{align*}
			u(X)=(px_1,px_2)~ \text{(mod deg.} \ 2) \ \text{and} \ u(X) \equiv (x_2^{h_1},x_1^{h_2}) \ (\text{mod} \ p),
		\end{align*}
		then the Jacobian matrix of $u(X)$ at $0$ is given by $$J_0(u)=\begin{pmatrix}
			p&0 \\ 0& p
		\end{pmatrix}~\text{and}~u(X) \not\equiv 0~(\text{mod}~p).$$
		This is a stable power series.
	\end{exmp}

\subsection{Galois Representations Attached to Simple Formal Groups} \label{ss2.1}

Let $F$ be a $d$-dimensional formal group of finite height $h$ over $\mathbb{Z}_p$, and let
\[
\operatorname{Tors}(F)=\bigcup_{n\ge1}F[p^n]
\]
denote the group of all $p$-power torsion points of $F$. The absolute Galois group
\[
G_{\mathbb{Q}_p}:=\operatorname{Gal}(\overline{\mathbb{Q}}_p/\mathbb{Q}_p)
\]
acts naturally on $\operatorname{Tors}(F)$ through its action on the coordinates of torsion
points. This action gives rise to a Galois representation
\[
\rho_F:G_{\mathbb{Q}_p}\longrightarrow \operatorname{Aut}(T_p(F)),
\]
where
\[
T_p(F)=\varprojlim_n F[p^n]
\]
denotes the Tate module of $F$. The corresponding $\mathbb{Q}_p$-vector space $V=T_pF \otimes_{\mathbb{Z}_p} \mathbb{Q}_p$ is a $G_{\mathbb Q_p}$-crystalline representation. 
	Then we have the following result:
	
	\begin{lem}  \label{lemma3.1}
		If $F$ is a $d$-dimensional simple formal group of height $h$, then the image of $\rho_F$ contains an open subgroup of $\mathbb{Z}_p^\times \cdot \operatorname{Id} \subset \operatorname{Aut}_{\mathbb{Z}_p}(T_pF)$.
	\end{lem}

\begin{proof}
Let $V := V_pF = T_pF \otimes_{\mathbb{Z}_p} \mathbb{Q}_p$.
Then we have a continuous $p$-adic representation
\[
\rho_F : G_{\mathbb{Q}_p} \longrightarrow \operatorname{Aut}(V)
\simeq \operatorname{GL}_h(\mathbb{Q}_p).
\]
Since $F$ is simple, $\rho_F$ is an irreducible 
$G_{\mathbb{Q}_p}$-representation: a nontrivial
$G_{\mathbb{Q}_p}$-stable subspace would correspond to a nontrivial
$p$-divisible subgroup of $F$, contradicting simplicity.
It is Hodge-Tate with weights 0 and 1 \cite{JT}. Let us pass to a suitable finite Galois extension $E$ of $\mathbb{Q}_p$ so that $\rho_F$ decomposes into a direct sum of absolutely irreducible representations, that is,
		\[\rho_F \otimes_{\mathbb{Q}_p} E=\bigoplus_{i=1}^r \rho_i\]
        
where each $\rho_i$ is absolutely irreducible. Fix one such constituent $\rho_i$ and rename it as $\rho_E:=\rho_i$.
By Schur's lemma applied to $\rho_E$ over $E$, the center of $\operatorname{Im}(\rho_E)$ consists exactly of scalar matrices $E^\times \cdot \operatorname{Id}$. 
On scalar matrices $\lambda \cdot \mathrm{Id}$, the determinant is given by
\[
\det(\lambda \cdot \mathrm{Id}) = \lambda^m,
\]
where $m=\dim \rho_E$. Thus the determinant induces a continuous homomorphism
\[
\det: \operatorname{Im}(\rho_E)\cap E^\times \cdot \operatorname{Id}
\longrightarrow \operatorname{Im}(\det(\rho_E)),
\]
which identifies with the map $\lambda \mapsto \lambda^m$ on $E^\times$, and whose image is an open subgroup of $E^\times$. In particular, this map has open image.

Thus, $\mathrm{Im}(\rho_F) \cap \mathbb{Z}_p^{\times} \cdot \mathrm{Id}$ is open if and only if $\mathrm{Im}(\det(\rho_E))$ contains an open subgroup of $\mathbb{Z}_p^{\times}$. 

So to prove our result, it suffices to show that $\mathrm{Im}(\det(\rho_E))$ contains an open subgroup of $\mathbb{Z}_p^{\times}$.

Since $\rho_F$ is Hodge-Tate, the character
\[
\chi:=\mathrm{det}(\rho_E): G_{\mathbb{Q}_p} \longrightarrow E^{\times}
\]
is also Hodge-Tate. By Artin's local reciprocity map, the character $\chi$ factors through abelianization $G_{\mathbb{Q}_p}^{\mathrm{ab}} \simeq \mathbb{Q}_p^\times$, and hence may be viewed as a character of $\mathbb{Q}_p^\times$. Under this identification, the restriction of $\chi$ to the inertia subgroup corresponds to its restriction to $\mathbb{Z}_p^{\times}$.
By \cite[Proposition~9.10]{FO}, the restriction of Hodge-Tate character $\chi$ to $\mathbb{Z}_p^{\times}$ is of the form
\[
\chi(x)=x^k \cdot \psi(x), \quad x \in \mathbb{Z}_p^{\times},\; k \in \mathbb{Z},
\]
where $\psi$ is a finite order character. Since $\rho_F$ has weights $0$ and $1$, the determinant has nonzero weight, hence $k \neq 0$. Since $\psi$ has finite image, it does not affect openness. 

Thus the image of $\chi$ contains $(\mathbb{Z}_p^{\times})^k$, which is a finite index (hence open) subgroup of $\mathbb{Z}_p^{\times}$. Therefore $\mathrm{Im}(\chi)=\mathrm{Im}(\mathrm{det}(\rho_E))$ contains an open subgroup of $\mathbb{Z}_p^{\times}$. This completes the proof.

\end{proof}
\begin{rem}
    We acknowledge that the idea of the above proof originates from a discussion on Mathoverflow \cite{Math}.
\end{rem}

\begin{cor} \label{c2.7.1}
Let $F$ be a simple formal group over $\mathbb{Z}_p$ and let
\[
\rho_F : G_{\mathbb{Q}_p} \longrightarrow 
\mathrm{Aut}_{\mathbb{Z}_p}(T_pF)
\]
be the associated Galois representation. Suppose $\sigma \in G_{\mathbb{Q}_p}$ satisfies
\[
\rho_F(\sigma)=a\cdot \mathrm{Id}
\quad\text{for some } a\in\mathbb{Z}_p^\times.
\]
Then for every torsion point $\xi \in \mathrm{Tors}(F)$ one has
\[
\sigma(\xi) = [a]_F(\xi).
\]
\end{cor}

\begin{proof}

By Lemma \ref{lemma3.1}, there exists $\sigma \in G_{\mathbb{Q}_p}$ satisfying
\[
\rho_F(\sigma)=a\cdot \mathrm{Id}
\quad\text{for some } a\in\mathbb{Z}_p^\times.
\]
An element of $T_pF$ is a compatible system $(\xi_n)_{n\ge1}$ with
\[
\xi_n \in F[p^n],
\qquad
[p]_F(\xi_{n+1})=\xi_n.
\]
The Galois action on $T_pF$ is defined component-wise:
\[
\sigma((\xi_n)_n)=(\sigma(\xi_n))_n.
\]
Given $\rho_F(\sigma)=a\cdot\mathrm{Id}$, the definition of $\rho_F$ implies that for every $(\xi_n)_n\in T_pF$,
\[
\sigma((\xi_n)_n)=a\cdot(\xi_n)_n = ([a]_F(\xi_n))_n,
\]
where multiplication by $a$ on $T_pF$ corresponds to the endomorphism $[a]_F$ on each level. 
Thus, if $\rho_F(\sigma)=a\cdot \mathrm{Id}$, then
\[
\sigma((\xi_n)_n)=([a]_F(\xi_n))_n.
\]
 Comparing components, we obtain
\[
\sigma(\xi_n) = [a]_F(\xi_n) \quad \text{for all } n.
\]

Now let $\xi \in \mathrm{Tors}(F)$. Then $\xi \in F[p^n]$ for some $n$, and it extends to a compatible system $(\xi_m)_m \in T_pF$. Applying the above identity gives
\[
\sigma(\xi) = \sigma(\xi_n) = [a]_F(\xi_n) = [a]_F(\xi).
\]
Hence $\sigma(\xi)=[a]_F(\xi)$.
Since $\xi$ was arbitrary, the equality holds for all torsion points of $F$. This proves the claim.
\end{proof}

\subsection{A Technical Result on Stable Endomorphisms}

	We look forward to prove an important result that says if a power series commutes with an endomorphism of a formal group, then it must itself be an endomorphism of that formal group.

	\begin{thm} \label{t2.11}
		Let $F$ be a $d$-dimensional commutative formal group over $\mathbb{Z}_p$, and let
		$u\in \operatorname{End}_{\mathbb{Z}_p}(F)$ be a \emph{stable} endomorphism.
		If
		\[
		w(X)\in \mathbb{Z}_p[[X]]^d,\qquad w(0)=0,
		\]
		satisfies $w\circ u = u\circ w$, then $w\in \operatorname{End}_{\mathbb{Z}_p}(F)$.
	\end{thm}
	
	\begin{proof}
		Since $F$ is a commutative formal group over $\mathbb{Z}_p$, it admits a logarithm
		\[
		L \colon F \longrightarrow \mathbb{G}_a^d
		\]
		defined over $\mathbb{Q}_p$, which is an isomorphism of formal groups. 
		
		The strategy is to transport the problem to the additive formal group $\mathbb{G}_a^d$ by conjugating all endomorphisms via the
		logarithm isomorphism $L$, establish the statement in that setting, and then
		transfer the conclusion back to $F$ using $L^{-1}$. Since $L$ is an isomorphism of formal groups, the desired property is preserved under this
		conjugation.
		
		So let us set
		\[
		\phi := L \circ u \circ L^{-1}, \qquad
		\theta := L \circ w \circ L^{-1}.
		\]
		Then $\phi$ is an endomorphism of the additive formal group
		$\mathbb{G}_a^d$, and the relation $w\circ u = u\circ w$ translates into
		\[
		\theta \circ \phi = \phi \circ \theta .
		\]
		We need to show $\theta$ is endomorphism of the additive formal group $\mathbb{G}_a^d$.

		Since $\phi$ is linear, we may write
		\[
		\phi(X)=MX
		\]
		for some $M\in \operatorname{GL}_d(\mathbb{Q}_p)$ whose eigenvalues are nonzero and not roots of unity
		(because $u$ is stable).
		We can write
		\[
		\theta(X)=BX+\sum_{m\ge2}E_m(X),
		\]
		where $B=J_0(\theta)\in M_d(\mathbb{Z}_p)$ and each $E_m(X)$ is a vector of homogeneous
		polynomials of total degree $m$.
		
		Comparing homogeneous components of degree $m\ge2$ in the identity
		$\theta(MX)=M\theta(X)$ gives
		\begin{equation}\label{eqq2.1}
			E_m(MX)=M E_m(X).
		\end{equation}
		
		Let $\mathcal{V}_m$ denote the $\mathbb{Q}_p$-vector space of homogeneous polynomial maps of total degree $m$, i.e.,
		\[\mathcal{V}_m:=\{E\colon \mathbb{Q}_p^d\to\mathbb{Q}_p^d~\mid~\text{$E$ is a homogeneous polynomial map ($d$-tuple vector) of degree $m$}  \}\]
		Write \[E(X)=(E_1(X), \cdots, E_d(X)), \]
		where each $E_i(X)$ is a homogeneous polynomial of degree $m$ in $X=(x_1, \cdots, x_d)$.
		Define following two operators $\Phi_m, \Psi_m$ on $\mathcal{V}_m$ by 
		
		\[
		\Phi_m(E):=E\circ M, \qquad \Psi_m(E):=M\circ E.
		\]
		We check that $\Phi_m$ is a linear operator. For,
		\begin{enumerate}
			\item[(i)] $E_1,E_2 \in \mathcal{V}_m$ implies $\Phi_m(E_1+E_2)=(E_1+E_2) \circ M=E_1 \circ M+E_2 \circ M=\Phi_m(E_1)+\Phi_m(E_2)$.
			\item[(ii)] for $a \in \mathbb{Q}_p$, $\Psi_m(aE_1)=M \circ (aE_1)=a (M \circ E_1)=a \Psi_m(E_1)$. 
		\end{enumerate}
		Similarly, $\Psi_m$ is also a linear operator. 
		Equation \eqref{eqq2.1} says precisely that
		\[
		\Phi_m(E_m)=\Psi_m(E_m).
		\]
		Choose a monomial basis $\{X^{\alpha} e_i \colon |\alpha|=m,~1 \leq i \leq m \}$, where $e_i$ is the $i$-th standard column vector.
		Let $\lambda_1,\dots,\lambda_d$ be the eigenvalues of $M$.
		The eigenvalues of $\Psi_m$ are exactly $\lambda_1,\dots,\lambda_d$. The linear operator $\Phi_m$ acts on $X^{\alpha}$ by multiplication with the scalar $\lambda^{\alpha}:=\lambda_1^{\alpha_1} \cdots \lambda_d^{\alpha_d}$. In other words, the eigenvalues
		of $\Phi_m$ are monomials
		\[
		\lambda_1^{\alpha_1}\cdots\lambda_d^{\alpha_d}
		\quad\text{with}\quad \alpha_1+\cdots+\alpha_d=m.
		\]
		Thus, the equation \eqref{eqq2.1} can only hold if
		\[\lambda^{\alpha} \in \{\lambda_1, \cdots, \lambda_d \}, \]
		which would force at least one eigen value $\lambda_i$ to be a root of unity or zero if $m \geq 2$. But the assumption of stable power series excludes that possibility. Hence,
		\[
		E_m(X)=0 \quad \text{for all } m\ge2.
		\]
		
		It follows that $\theta(X)=BX$ is linear, hence an endomorphism of the additive
		formal group.
		Transporting back via the logarithm yields
		\[
		w=L^{-1}\circ \theta\circ L \in \operatorname{End}_{\mathbb{Z}_p}(F),
		\]
		which completes the proof.
	\end{proof}

	\subsection{Rigidity Results}

	We prove the following result using the same strategy as before, namely, we conjugate the relevant endomorphisms by the logarithm in order to transport the problem to the additive formal group $\mathbb{G}_a^d$.
	
	\begin{thm} \label{t2.12}
		Let $F$ and $G$ be two $d$-dimensional commutative formal groups over $\mathbb{Z}_p$.
		Suppose there exists a power series
		\[
		u(X) \in \mathbb{Z}_p[[X]]^d, \qquad u(0)=0,
		\]
		such that
		\[
		u \in \operatorname{End}_{\mathbb{Z}_p}(F) \cap \operatorname{End}_{\mathbb{Z}_p}(G).
		\]
		Assume further that $u$ is \emph{stable}, then $F=G$.
	\end{thm}
	
	\begin{proof}
       Let
		\[
		L_F : F \xrightarrow{\;\sim\;} \mathbb{G}_a^d,
		\qquad
		L_G : G \xrightarrow{\;\sim\;} \mathbb{G}_a^d
		\]
		be the logarithms of $F$ and $G$, respectively. These are formal group
		isomorphisms over $\mathbb{Q}_p$, satisfying
		\[
		L_F(F(X,Y)) = L_F(X)+L_F(Y),
		\qquad
		L_G(G(X,Y)) = L_G(X)+L_G(Y).
		\]
		
		Define the conjugates of $u$ by
		\[
		\phi_F := L_F \circ u \circ L_F^{-1},
		\qquad
		\phi_G := L_G \circ u \circ L_G^{-1}.
		\]
		Then $\phi_F,\phi_G \in \mathbb{Q}_p[[X]]^d$, $\phi_F(0)=\phi_G(0)=0$. Moreover,
		\[
		J_0(\phi_F) = J_0(\phi_G) = J_0(u),
		\]
		because $\phi_F=L_F \circ u \circ L_F^{-1} \Rightarrow J_0(\phi_F)=J_0(L_F) J_0(u)J_0(L_F^{-1})=J_0(u)$, since $J_0(L_F)=\operatorname{Id}=J_0(L_F^{-1})$.
		In particular, both $\phi_F$ and $\phi_G$ are stable endomorphisms of the additive formal group $\mathbb{G}_a^d$.

	 Since $\phi_F$ and $\phi_G$ are linear and have the same Jacobian at the origin, we must have
		\[
		\phi_F = \phi_G =: \phi.
		\]
		
		Define
		\[
		H := L_G \circ L_F^{-1} \in \mathbb{Q}_p[[X]]^d.
		\]
		Then $H(0)=0$ and $J_0(H)=\mathrm{Id}$. Moreover, we have
		\begin{align}
			\nonumber H \circ \phi&=(L_G \circ L_F^{-1}) \circ \phi \\
			 \nonumber &=(L_G \circ L_F^{-1}) \circ (L_F \circ u \circ L_F^{-1}),~~(\text{since}~\phi=\phi_F= L_F \circ u \circ L_F^{-1})  \\
			 \nonumber &=L_G \circ (L_F^{-1} \circ L_F) \circ u \circ L_F^{-1} \\
		\label{eq2.2}	 &=L_G \circ u \circ L_F^{-1}
		\end{align}
		\begin{align}
			\nonumber \phi \circ H&=\phi \circ (L_G \circ L_F^{-1}) \\
			\nonumber &=(L_G \circ u \circ L_G^{-1}) \circ (L_G \circ L_F^{-1}),~~(\text{since}~\phi=\phi_G= L_G \circ u \circ L_G^{-1})   \\
			\nonumber &=L_G \circ u \circ (L_G^{-1} \circ L_G) \circ L_F^{-1} \\
			\label{eq2.3}	 &=L_G \circ u \circ L_F^{-1}.
		\end{align}
		From the equations \eqref{eq2.2} and \eqref{eq2.3}:
		\[H \circ \phi = \phi \circ H.\]
		
		Because $\phi$ is an stable endomorphism of $\mathbb{G}_a^d$, applying Theorem \ref{t2.11}
		implies that $H$ is also an endomorphism of $\mathbb{G}_a^d$. Therefore,
		\[
		H(X+Y) = H(X)+H(Y),
		\]
		and $H$ is linear. As $J_0(H)=\mathrm{Id}$, it follows that $H=\mathrm{Id}$.
		
		Since $L_G \circ L_F^{-1}=H=\mathrm{Id}$, we have $L_F=L_G$. Therefore,
		\[
		F(X,Y)
		= L_F^{-1}(L_F(X)+L_F(Y))
		= L_G^{-1}(L_G(X)+L_G(Y))
		= G(X,Y).
		\]
		Hence $F=G$, as claimed.
	\end{proof}

	\begin{lem} \label{l3.5}
		If $F$ is a $d$-dimensional simple formal group over $\mathbb{Z}_p$ and if $h(X) \in \mathbb{Z}_p[[X]]^d$ is such that $h(0)=0$ and $h(Z) \in \text{Tors}(F)$ for all $Z \in \mathcal{Z}$, where $\mathcal{Z} \subset \text{Tors}(F)$ is Zariski dense, then $h \in \text{End}_{\mathbb{Z}_p}(F)$.
	\end{lem}
	\begin{proof}
		By Corollary \ref{c2.7.1}, there exists $\sigma \in \text{Gal}(\bar{\mathbb{Q}}_p/\mathbb{Q}_p)$ and a stable endomorphism $u$ of $F$ with
		\begin{align} \label{e3}
			\sigma(Z)=u(Z)~\text{for all~} Z \in \text{Tors}(F).
		\end{align}
		Since $h(Z) \in \text{Tors}(F)$, for every $Z \in \mathcal Z$, substituting $Z$ by $h(Z)$ in equation \eqref{e3} yields
		
		\begin{align} \label{e4}
			\sigma(h(Z))=u(h(Z))~\text{for all~} Z \in \text{Tors}(F).
		\end{align}
		Moreover, according to equation \eqref{e3}, $\sigma(h(Z))=h(\sigma(Z))=h(u(Z)$ holds true for $Z \in \mathcal Z$. Thus, we derive from equation \eqref{e4}
		\begin{align*}
			&u(h(Z))=h(u(Z))~\text{for all~} Z \in \mathcal{Z} \\
			\Rightarrow & u \circ h-h \circ u=0~\text{on}~ \mathcal{Z}.
		\end{align*}
		Because $\mathcal Z$ is Zariski dense subset of $\mathfrak{m}_{\mathbb{C}_p}^d$, the rigidity of $p$-adic power series implies $u \circ h=h \circ u$. Therefore, by Theorem \ref{t2.11}, we conclude $h \in \text{End}_{\mathbb{Z}_p}(F)$.
	\end{proof}

	\subsection{Proof of the Main Theorem}
	Now, we are in a position to prove the main Theorem \ref{t3.6} by utilizing the aforementioned technical and rigidity results that we just developed.
	\begin{proof}[Proof of Theorem \ref{t3.6}]
			Let $F$ and $G$ be $d$-dimensional simple formal groups over $\mathbb{Z}_p$, and assume that
		\[
		\mathcal Z := \operatorname{Tors}(F) \cap \operatorname{Tors}(G)
		\]
		is infinite. Consider the $p$-divisible groups associated to $F$ and $G$:
		\[
		F[p^\infty] := \varprojlim_n F[p^n], \qquad
		G[p^\infty] := \varprojlim_n G[p^n].
		\]
		For each $n$, the intersection $F[p^n]\cap G[p^n]$ is a finite flat subgroup scheme of $F[p^n]$.  
		These intersections are compatible under the multiplication $p$, because $F[p^{n+1}]\to F[p^n]$ restrict to $H[p^{n+1}] \to H[p^n]$, 
		and similarly for $G$. Therefore,
		\[
		H := F[p^\infty] \cap G[p^\infty] := \varprojlim_n (F[p^n]\cap G[p^n])
		\]
		is a $p$-divisible subgroup of both $F[p^\infty]$ and $G[p^\infty]$, of finite height.

		By the equivalence between finite-height formal groups and connected p-divisible groups \cite[Proposition~1]{JT}, the simplicity of \(F\) and \(G\) implies that the \(p\)-divisible groups \(F[p^\infty]\) and \(G[p^\infty]\) are simple. Therefore, the common \(p\)-divisible subgroup
\[
H \subseteq F[p^\infty] \quad \text{and} \quad H \subseteq G[p^\infty]
\]
must be either trivial or equal to the whole \(p\)-divisible group. Since
\[
Z=\operatorname{Tors}(F)\cap\operatorname{Tors}(G)
\]
is infinite, \(H\) is nontrivial. Hence,
\[
H=F[p^\infty]=G[p^\infty].
\]

		In particular, the intersection $\mathcal Z = \operatorname{Tors}(F)\cap \operatorname{Tors}(G)$ is Zariski dense
		in $\mathfrak m_{\mathbb{C}_p}^d$.

		By Corollary~\ref{c2.7.1}, there exists $\sigma \in \operatorname{Gal}(\overline{\mathbb{Q}}_p/\mathbb{Q}_p)$ and a nonzero endomorphism $u \in \operatorname{End}(F)$
		such that
		\begin{equation}\label{eq2.6}
			\sigma(Z) = u(Z) \quad \text{for all } Z \in \operatorname{Tors}(F).
		\end{equation}
		
		The set $\mathcal Z$ is stable under the action of $\text{Gal}(\bar{\mathbb{Q}}_p/\mathbb{Q}_p)$, so for $Z \in \mathcal Z$ we have
		\[
		\sigma(Z) \in \mathcal Z \subset \operatorname{Tors}(G).
		\]
		By \eqref{eq2.6}, it follows that
		\[
		u(Z) \in \operatorname{Tors}(G)
		\quad \text{for all } Z \in \mathcal Z.
		\]
		
		Since $\mathcal Z$ is Zariski dense in $\mathfrak m_{\mathbb{C}_p}^d$, Lemma~\ref{l3.5} implies that
		$u \in \operatorname{End}(G)$.
		Thus $u$ is a stable endomorphism of both $F$ and $G$:
		\[
		u \circ F = F \circ u, \qquad u \circ G = G \circ u.
		\]
		By Theorem~\ref{t2.12}, a $d$-dimensional formal group over $\mathbb{Z}_p$ is uniquely
		determined by a stable endomorphism. Applying this to $u$, we conclude that $F = G$.
	\end{proof}

	\section{Conclusion and Future Direction} \label{s4}
	In this article, our observation is that the $p$-power torsion points of a formal group can identify that formal group, and in this regard we have proved Theorem \ref{t3.6}. 
	
	However, in our setting, we considered \textit{simple} formal groups otherwise, the result might not hold true. For, one can consider $F=H \times F_1$ and $G=H \times G_1$, where $H$ is a $(d-1)$-dimensional formal group and $F_1,G_1$ are distinct $1$-dimensional formal groups. In that case, $\text{Tors}(F) \cap \text{Tors}(G)=\text{Tors}(H) \times \{0\}$, which is infinite, but $F$ and $G$ are not equal. This naturally leads to the following question.
  \begin{question}
Under what additional assumptions, weaker than simplicity, does an infinite intersection
\[
\operatorname{Tors}(F)\cap\operatorname{Tors}(G)
\]
imply that \(F\) and \(G\) are isomorphic? Furthermore, under what additional conditions can one conclude that \(F=G\)?
\end{question}

    Theorem~\ref{t3.6} suggests the possibility of extending the rigidity phenomenon from formal groups to abelian varieties. This motivates the following question.
\begin{question}
Let $A$ and $B$ be $d$-dimensional simple abelian varieties admitting integral models over $\mathbb{Z}_p$. For each $n\ge1$, let $A[p^n]$ and $B[p^n]$ denote the groups of $p^n$-torsion points, and define
\[
\operatorname{Tors}(A)=\bigcup_{n\ge1}A[p^n], \qquad
\operatorname{Tors}(B)=\bigcup_{n\ge1}B[p^n].
\]
If $\operatorname{Tors}(A)\cap\operatorname{Tors}(B)$
is infinite, does it follow that $A$ and $B$ are isogenous?
\end{question}
	
	\subsection*{Acknowledgement} The authors are deeply grateful to Professor Jonathan Lubin, who offered many helpful feedback, in particular sharing an idea for the proof of Theorem \ref{t2.11}. The authors acknowledge Professor Joseph H. Silverman for a helpful discussion showing Zariskiness of the set $\text{Tors}(F) \cap \text{Tors}(G)$ in the proof of Theorem \ref{t3.6}. The authors would like to express their sincere gratitude to Vicen\c{t}iu Pa\c{s}ol and Mugurel Barcau for their helpful comments and assistance during the revision of this article. The first author is grateful to \textit{CSIR}, Govt. of India, for the grant with File no.-09/025(0249)/2018-EMR-I.

	\appendix
	\section{} \label{A1}
	\begin{proof}[Proof of Proposition \ref{pp2.8}]
		Because $h(X)$ is invertible, there must exist some inverse $h^{-1}(X) \in \mathcal{O}_K[[X]]^d$ for $h^{-1}(0)=0$ so that $h(h^{-1}(X))=X$. By chain rule, $J(h(h^{-1}(X))=J(h)(X)J(h^{-1})(X)=J(X)$ such that evaluating at $0$ yields $J_0(h)J_0(h^{-1})=I_d$, the $d \times d$ identity matrix $I_d$. As a result, both $J_0(h),J_0(h^{-1}) \in \text{GL}_d(\mathcal{O}_K^{\times})$.
		
		Conversely,  suppose $J_0(h) \in \text{GL}_d(\mathcal{O}_K^{\times})$. The construction of an inverse $h^{-1}(X)$ is sufficient. Induction will be used. If we use the linear algebraic expression $f^1(X)=J_0(h^{-1})X \in \mathcal{O}_K[[X]]^d$, we get
		$$h(f^1(X)) \equiv X~(\text{mod deg 2}).$$ 
		Suppose the polynomial $f^n(X) \in \mathcal{O}_K[[X]]^d$ with $f^n(X) \equiv f^{n-1}(X)$ (mod deg $n$) fulfilling $$h(f^n(X)) \equiv X~(\text{mod deg n+1})$$ for $n \geq 1$. The following polynomial $f^{n+1}(X) \in \mathcal{O}_K[[X]]^d$ must then have the formula $$f^{n+1}(X)=f^{n}(X)+R^{n+1}(\text{remainder}),$$ where $R^{n+1}(X) \equiv 0$ (mod deg $n+1$) and $R^{n+1}$ is actually a homogeneous polynomial with degree $n+1$.
		Since $J_0(h)$ is invertible, we can solve for $R^{n+1}$ from the condition $$h(f^{n+1}(X))=h(f^n(X)+R^{n+1})) \equiv h(f^n(X))+J_0(h)R^{n+1}(X) \equiv X~(\text{mod deg}~n+2)$$ giving us $R^{n-1} \equiv J_0(h)^{-1} (X-h(f^n(X))) \equiv 0$ (mod deg $n+1$). So we derive the inverse of $h(X)$ via induction on $f^n(X)$. The proof is now complete.
	\end{proof}

\end{document}